\documentclass{amsart}
\usepackage{amsmath,amsthm,amscd,amssymb}
\usepackage[mathscr]{eucal}
\usepackage{amssymb}
\usepackage{latexsym}

\theoremstyle{definition}
\newtheorem{Def}{Definition}[section]
\newtheorem{Thm}[Def]{Theorem}

\newtheorem{Rem}[Def]{Remark}
\newtheorem{Ex}[Def]{Example}
\newtheorem{Cor}[Def]{Corollary}

\newtheorem{Tab}[Def]{Table}

\numberwithin{equation}{section}

\begin{document}

\title{On theta series attached to the Leech lattice}

\author{Shoyu Nagaoka and Sho Takemori}
\maketitle



\noindent
%
\begin{abstract}
Some congruence relations satisfied by the theta series associated with the Leech lattice are
given. 

\end{abstract}
\maketitle

\section{Introduction}
The main object of this note is concerned with the degree $n$ theta series
$\vartheta_{\varLambda}^{(n)}$ associated with a rank $m$ unimodular lattice $\varLambda=\varLambda_m$
(or equivalently, degree $m$ symmetric matrix $S=S_m$). It is known that the theta series 
$\vartheta_{\varLambda}^{(n)}$ becomes a Siegel modular form of weight $\frac{m}{2}$
for $\Gamma_n:=Sp_n(\mathbb{Z})$. Now we consider the case that $\varLambda=\mathcal{L}$: 
the Leech lattce.
The Leech lattice is an even unimodular 24-dimensional lattice (cf. $\S$ \ref{thetaLeech}). 
Therefore $\vartheta_{\mathcal{L}}^{(n)}$
is a Siegel modular form of weight 12 for $\Gamma_n$. The main purpose is to show that
the theta series $\vartheta_{\mathcal{L}}^{(2)}$ satisfies the following congruence relation:
\[
\varTheta(\vartheta_{\mathcal{L}}^{(2)}) \equiv 0 \pmod{23}\qquad ({\rm Theorem}\; \ref{mainTH2}).
\]
where $\varTheta$ is the theta operator defined in $\S$ \ref{thetaoperator}.
In this note, we present two different kinds of the proof. 
The first proof depends on the fact that the image of $\varTheta(\vartheta_{\mathcal{L}}^{(2)})$
under the theta operator is congruent to a cusp form (cf. $\S$ \ref{first}). The second proof
is based on the fact that $\varTheta(\vartheta_{\mathcal{L}}^{(2)})$ is congruent to the theta
series associated with the binary quadratic form of discriminant $-23$ (cf. $\S$ \ref{second}).

It should be noted that a similar congruence relation appeared in \cite{K-K-N}. 
That is, the following congruence relation was proved:
\[
\varTheta(X_{35}) \equiv 0 \pmod{23}.
\]
where $X_{35}$ is the Igusa cusp form of weight 35 (cf. $\S$ \ref{generators}).

In $\S$ 3.5, we introduce another congruence relation satisfied by $\vartheta_{\mathcal{L}}^{(2)}$:
\[
\vartheta_{\mathcal{L}}^{(2)} \equiv 1 \pmod{13}\qquad ({\rm Theorem}\; \ref{anothercongruence}).
\]
This gives an example of weight $p-1$ modular form $F$ satisfying
$F \equiv 1 \pmod{p}$ (e.g.\,cf.\;\cite{B-N-1}).

\section{Preliminaries}
\subsection{Notation}
\label{notation}
First we confirm notation. Let $\Gamma _n=Sp_n(\mathbb{Z})$ be the Siegel modular group of degree $n$ and $\mathbb{H}_n$ 
the Siegel upper-half space of degree $n$. We denote by
$M_k(\Gamma _n)$ the $\mathbb{C}$-vector space of all
Siegel modular forms of weight $k$ for $\Gamma_n$, and $S_k(\Gamma _n)$ is the subspace of cusp forms.

Any $F(Z)$ in $M_k(\Gamma _n)$ has a Fourier expansion of the form
\[
F(Z)=\sum_{0\leq T\in Sym_n^*(\mathbb{Z})}a(F;T)q^T,\quad q^T:=\text{exp}(2\pi i\text{tr}(TZ)),
\quad Z\in\mathbb{H}_n,
\]
where
\[
Sym_n^*(\mathbb{Z}):=\{ T=(t_{lj})\in Sym_n(\mathbb{Q})\;|\; t_{ll},\;2t_{lj}\in\mathbb{Z}\; \}.\\
\]
Namely we write the Fourier coefficient corresponding to $T\in Sym_n^*(\mathbb{Z})$ as $a(F;T)$.

For a subring $R$ of $\mathbb{C}$, let $M_{k}(\Gamma _n)_{R}\subset M_{k}(\Gamma _n)$ denote 
the space of all modular forms 
whose Fourier coefficients lie in $R$. 
\subsection{Formal Fourier expansion}
\label{FormalFourier}
For $T=(t_{lj})\in Sym_n^*(\mathbb{Z})$ and $Z=(z_{lj})\in\mathbb{H}_n$, we write 
$q_{lj}:=\text{exp}(2\pi iz_{lj})$. Then
\[
q^T=\text{exp}(2\pi i\text{tr}(TZ))=\prod_{l<j}q_{lj}^{2t_{lj}}\prod_{l=1}^nq_{ll}^{t_{ll}}.
\]
Therefore we may consider $F\in M_k(\Gamma_n)_R$ as an element of the formal power series ring:
\[
F=\sum a(F;T)q^T\in R[q_{lj},q_{lj}^{-1}][\![q_{11},\ldots ,q_{nn}]\!].
\]
For a prime $p$, we denote by $\mathbb{Z}_{(p)}$ the local ring at $p$. For two elements
\[
F_i=\sum a(F_i;T)q^T\in \mathbb{Z}_{(p)}[q_{lj},q_{lj}^{-1}][\![ q_{11},\ldots ,q_{nn}]\!],\;(i=1,\,2),
\]
we write $F_1 \equiv F_2 \pmod{p}$ if the congruence
\[
a(F_1;T) \equiv a(F_2;T) \pmod{p}
\]
satisfies for all $T\in Sym_n^*(\mathbb{Z})$.
\subsection{Theta operator}
\label{thetaoperator}
For a $F=\sum a(F;T)q^T\in M_k(\Gamma_n)$, we associate the formal power series
\[
\varTheta (F):=\sum a(F;T)\cdot\text{det}(T) q^T
\in\mathbb{C}[q_{lj},q_{lj}^{-1}][\![q_{11},\ldots ,q_{nn}]\!].
\]
It should be noted that $\varTheta (F)$ is not necessarily modular form. However the following
fact holds.
\begin{Thm} (B\"{o}cherer-Nagaoka \cite{B-N-1}, Theorem 4).
\label{B-N-1}
Assume that a prime $p$ satisfies $p\geq n+3$.
Then, for any modular form $F\in M_k(\Gamma_n)_{\mathbb{Z}_{(p)}}$, there
exists a Siegel cusp form $G\in S_{k+p+1}(\Gamma_n)_{\mathbb{Z}_{(p)}}$ satisfying
\[
\varTheta (F) \equiv G \pmod{p}
\]
as formal power series.
\end{Thm}
In the case $n=1$, this operator was studied by Ramanujan and played an important role
in the theory of $p$-adic elliptic modular forms (\cite{Serre}).
\subsection{Theta series and Leech lattice}
\label{thetaLeech}
As usual, for a positive matrix $S\in 2Sym_m^*(\mathbb{Z})$, we associate the theta series on $\mathbb{H}_n$:
\[
\vartheta_S^{(n)}(Z):=\sum_{X\in M_{m,n}(\mathbb{Z})}\text{exp}(\pi i\text{tr}(S[X]Z))
\]
where $S[X]={}^t\!XSX$. It is well-known that
\[
\vartheta_S^{(n)}\in M_{\frac{m}{2}}(\Gamma_n)_{\mathbb{Z}}
\]
if $S$ is unimodular.

Let $\mathcal{L}$ be the Leech lattice (we identify it with the Gram matrix). The Leech lattice is
the unique lattice of rank 24, which contains no roots (e.g. cf. \cite{Ebeling}, Theorem 4.1). From this fact,
for example. we see that
\[
a\left( \vartheta_{\mathcal{L}}^{(2)};
\begin{pmatrix}m & \frac{r}{2} \\ \frac{r}{2} &1 \end{pmatrix} \right)
=
a\left( \vartheta_{\mathcal{L}}^{(2)};
\begin{pmatrix}1 & \frac{r}{2} \\ \frac{r}{2} &n \end{pmatrix} \right)
=0.
\]
\begin{Ex}
\label{Leechdegree1}
It is known that
\begin{align*}
\vartheta_{\mathcal{L}}^{(1)}&= (E_4^{(1)})^3-720\varDelta\\
&=1+196560q^2+16773120q^3+398034000q^4+\cdots \in M_{12}(\Gamma_1)_{\mathbb{Z}},\\
& (\text{e.g.\;cf.\;\cite{Skoruppa}}),
\end{align*}
where $\varDelta$ is the cusp form of weight 12 defined by
\begin{align*}
\varDelta &:=\frac{1}{1728}((E_4^{(1)})^3-(E_6^{(1)})^2)\\
           &=q-24q^2+252q^3-1472q^4+4830q^5-6048q^6-\cdots
           \in S_{12}(\Gamma_1)_{\mathbb{Z}}.
\end{align*}
\end{Ex}
\subsection{Sturm-type Theorem}
\label{sturm}
A Sturm-type theorem maintains that, if some of Fourier coefficients of $F$
vanish mod $p$, then $F$ is congruent zero mod $p$.

In the following, for simplicity, we use the abbreviation
\[
[m,r,n]:=
\begin{pmatrix}
m & \frac{r}{2}\\
\frac{r}{2} & n
\end{pmatrix},
\qquad
(m,\,n,\,r\in\mathbb{Z}).
\]

\begin{Thm}(D.~Choi, Y.~Choie, T.~Kikuta \cite{C-C-K}).
\label{C-C-K}
Let $p\geq 5$ be a prime. Suppose that $F(Z)\in M_k(\Gamma_2)_{\mathbb{Z}_{(p)}}$ has
the Fourier expansion
\[
F(Z)=\sum_{0\leq T=[m,r,n]}a([m,r,n])q^T.
\]
If
\[
a([m,r,n]) \equiv 0 \pmod{p}
\]
for any $m$, $n$ such that
\[
0\leq m\leq \frac{k}{10}\quad \text{and}\quad 0\leq n\leq \frac{k}{10},
\]
then $F \equiv 0 \pmod{p}$.
\end{Thm}
\begin{Rem}
In \cite{C-C-K}, the result was proved under more general situation.
\end{Rem}
\begin{Cor}
\label{weight12}
Suppose that $F\in M_{12}(\Gamma_2)_{\mathbb{Z}}$ satisfies
\[
a(F;T)=0
\]
for any $0\leq T\in Sym_2^*(\mathbb{Z})$ with $\text{tr}(T)\leq 2$, then $F=0$.
\end{Cor}
\begin{proof}
We can apply Theorem \ref{C-C-K} to the case $k=12$ and infinitely many $p$.
\end{proof}
\begin{Cor}
\label{weight36}
Suppose that $F\in M_{36}(\Gamma_2)_{\mathbb{Z}}$ satisfies
\[
a(F;T) \equiv 0 \pmod{23}
\]
for any $0\leq T\in Sym_2^*(\mathbb{Z})$ with $\text{tr}(T)\leq 6$, then
$F \equiv 0 \pmod{23}$.
\end{Cor}
\subsection{Congruences for binary theta series}
\label{Congbinary}
Assume that $p$ is a prime with $p \equiv 3 \pmod{4}$. Then there exists
$S\in Sym_2^*(\mathbb{Z})$ such that
\[
\text{det}(2S)=p,
\]
namely, the discriminant of $S$ is $-p$. For this $S$, we associate the theta
series $\vartheta_S^{(2)}$. Then $\vartheta_S^{(2)}$ is a weight $1$ modular form
for
\[
\Gamma_0^2(p):=\left\{\;\begin{pmatrix}A & B\\ C& D  \end{pmatrix}\in\Gamma_2\;|\;
C \equiv 0_2 \pmod{p}\; \right\}
\]
with character $\chi_p=\left( \frac{-p}{*} \right)$. Namely
\[
\vartheta_S^{(2)}\in M_1(\Gamma_0^2(p),\chi_p).
\]
The following statement is a special case of a result of B\"{o}cherer and Nagaoka 
(cf. \cite{B-N-2}, Theorem 5).
\begin{Thm}(S.~B\"{o}cherer, S.~Nagaoka)
\label{B-N-2}
Assume that $p\geq 7$ and $p \equiv 3 \pmod{4}$. Let $S\in Sym_2^*(\mathbb{Z})$ be a positive
definite binary quadratic form with $\text{det}(2S)=p$ (i.e. discriminant of $S=-p$.)
Then there exists a modular form $G\in M_{\frac{p+1}{2}}(\Gamma_2)_{\mathbb{Z}_{(p)}}$
such that
\[
\vartheta_S^{(2)} \equiv G \pmod{p}.
\]
\end{Thm}
\section{Main result}
\label{main}
\subsection{Statement of the main result}
As we stated in Introduction, the main purpose of this note is to show that the
theta series associated with the Leech lattice satisfies a congruence relation.
\begin{Thm}
\label{mainTH}
Let $a(\vartheta_{\mathcal{L}}^{(2)};T)$ denote the Fourier coefficient of $\vartheta_{\mathcal{L}}^{(2)}$.
If ${\rm det}(T)\not\equiv 0 \pmod{23}$, then
\[
a(\vartheta_{\mathcal{L}}^{(2)};T) \equiv 0 \pmod{23},
\]
or equivalently,
\[
\varTheta (\vartheta_{\mathcal{L}}^{(2)}) \equiv 0 \pmod{23}.
\]
\end{Thm}
\subsection{The first proof}
\label{first}
In this subsection we present a proof of Theorem \ref{mainTH} using a property of the theta operator.
\begin{proof}
We apply Theorem \ref{B-N-1} in the case that 
\[
F=\vartheta_{\mathcal{L}}^{(2)}\qquad\text{and}\qquad p=23.
\]
From this, we can find a Siegel cusp form $G\in S_{36}(\Gamma_2)$ such that
\[
\varTheta (\vartheta_{\mathcal{L}}^{(2)}) \equiv G \pmod{23}.
\]
By the Table \ref{coefficients} in $\S 4$, we can confirm
\[
a(\varTheta (\vartheta_{\mathcal{L}}^{(2)});T) \equiv a(G;T) \equiv 0 \pmod{23}
\]
for any $0\leq T\in Sym_2^*(\mathbb{Z})$ with $\text{tr}(T)\leq 6$. Then, by Corollary 
\ref{weight36}, we obtain
\[
G \equiv 0 \pmod{23}.
\]
This means that
\[
\varTheta (\vartheta_{\mathcal{L}}^{(2)}) \equiv 0 \pmod{23}.
\]
\end{proof}
\subsection{The second proof}
\label{second}
In this subsection, we give the second proof of our main theorem, which is based on a
congruence between theta series.

\begin{Thm}
\label{mainTH2}
The following congruence relation holds.
\[
\vartheta_{[2,1,3]}^{(2)} \equiv \vartheta_{\mathcal{L}}^{(2)}
\pmod{23},
\]
or equivalently,
\[
\varTheta (\vartheta_{\mathcal{L}}^{(2)}) \equiv 0 \pmod{23}.
\]
\end{Thm}
\begin{proof}
We apply Theorem \ref{B-N-2} in the case $p=23$. 
Then we see that there is a modular form $G\in M_{12}(\Gamma_2)_{\mathbb{Z}_{(23)}}$
such that
\[
\vartheta_{[2,1,3]}^{(2)} \equiv G \pmod{23}.
\]
By the Tables \ref{coefficients} in $\S 4$, we can confirm that
\[
a(\vartheta_{[2,1,3]}^{(2)};T) \equiv a(G;T) \equiv a(\vartheta_{\mathcal{L}}^{(2)};T)\pmod{23}
\]
for any $0\leq T\in Sym_2^*(\mathbb{Z})$ with $\text{tr}(T)\leq 6$. This shows that
\[
G \equiv \vartheta_{\mathcal{L}}^{(2)} \pmod{23}.
\]
Since
\[
\varTheta (\vartheta_{[2,1,3]}^{(2)}) \equiv 0 \pmod{23},
\]
we obtain
\[
\varTheta (\vartheta_{\mathcal{L}}^{(2)}) \equiv 0 \pmod{23}.
\]
\end{proof}
\begin{Rem}
In the degree one case, we have already known the congruence
\[
\vartheta_{[2,1,3]}^{(1)} \equiv \vartheta_{\mathcal{L}}^{(1)}
\pmod{23},\qquad 
\]
(\cite{Skoruppa}, \,p.3).
\end{Rem}
\subsection{Igusa's generators}
\label{generators}
The theta series $\vartheta_{\mathcal{L}}^{(2)}$ is a weight 12 Siegel modular form
with integral Fourier coefficients. Therefore it can be expressed as a polynomial 
with Igusa's generators of the ring of Siegel modular forms of degree two over
$\mathbb{Z}$. In this subsection, we give the explicit form.

Let
\[
M(\Gamma_2)_{\mathbb{Z}}=\bigoplus_{k\in\mathbb{Z}}M_k(\Gamma_2)_{\mathbb{Z}}.
\]
be the graded ring of Siegel modular forms of degree 2 over $\mathbb{Z}$.
Igusa \cite{Igusa} consructed a minimal set of generators of this ring.
The set of generators consists of fifteen modular forms
\[
X_4,\,X_6,\,X_{10},\,Y_{12},\, X_{12},\ldots,X_{48}
\]
where subscripts denote their weights. Here the first two modular forms $X_k\,(k=4,6)$
are the weight $k$ Siegel-Eisenstein series:
\[
X_4=E_4^{(2)},\quad X_6=E_6^{(2)}.
\]
The modular form $X_{35}$ appearing in Introduction is one of these generators, moreover, it is
the unique odd weight generator.
\begin{Ex}
\label{Igusagenerator}
We give the Fourier expansions of the first five generators:\\
\begin{align*}
X_4 &=1+240(q_{11}+q_{22})+2160(q_{11}^2+q_{22}^2)+(30240+13440\cdot c_1+240\cdot c_2)q_{11}q_{22}\\
&+6720(q_{11}^3+q_{22}^3)+(181440+138240\cdot c_1+30240\cdot c_2)(q_{11}^2q_{22}+q_{11}q_{22}^2)+\cdots
\\
X_6 &=1-504(q_{11}+q_{22})-16632(q_{11}^2+q_{22}^2)+(166320+44352\cdot c_1-504\cdot c_2)q_{11}q_{22}\\
&-122976(q_{11}^3+q_{22}^3)+(3792096+2128896\cdot c_1+166320\cdot c_2)(q_{11}^2q_{22}+q_{11}q_{22}^2)+\cdots
\\
X_{10} &=(-2+c_1)q_{11}q_{22}+(36-16\cdot c_1-2\cdot c_2)(q_{11}^2q_{22}+q_{11}q_{22}^2)\\
&+(-272+99\cdot c_1+36\cdot c_2+c_3)(q_{11}^3q_{22}+q_{11}q_{22}^3)+\cdots
\\
X_{12} &=(10+c_1)q_{11}q_{22}+(-132-88\cdot c_1+10\cdot c_2)(q_{11}^2q_{22}+q_{11}q_{22}^2)\\
&+(736+1275\cdot c_1-132\cdot c_2+c_3)(q_{11}^3q_{22}+q_{11}q_{22}^3)+\cdots
\\
Y_{12} &=(q_{11}+q_{22})-24(q_{11}^2+q_{22}^2)+(1206+116\cdot c_1+c_2)q_{11}q_{22}\\
&+252(q_{11}^3+q_{22}^3)+(115236+22176\cdot c_1+1206\cdot c_2)(q_{11}^2q_{11}+q_{11}q_{22}^2)+\cdots,
\end{align*}
where $c_i=q_{12}^i+q_{12}^{-i}$. (We have more extended expression for each modular form.)
\end{Ex}
Here we should remark that
\[
\varPhi(X_4)=E_4^{(1)},\;\;
\varPhi(X_6)=E_6^{(1)},\;\;
\varPhi(X_{10})=\varPhi(X_{12})=0,\;\;
\varPhi (Y_{12})=\varDelta,
\]
where $\varPhi$ is the Siegel operator.
\begin{Thm}
\label{mainTH3}
Let $\vartheta_{\mathcal{L}}^{(2)}$ is the degree 2 theta series associated with the
Leech lattice $\mathcal{L}$. The we have
\[
\vartheta_{\mathcal{L}}^{(2)}=X_4^3-720Y_{12}+43200X_{12}.
\]
\end{Thm}
\begin{proof}
As a matter of course, Theorem \ref{mainTH3} can be obtained by the direct calculations
of the Fourier coefficients of $X_4$, $X_6$, $X_{12}$, and $Y_{12}$. By Igusa's structure
theorem over $\mathbb{Z}$, we can warite as
\[
\vartheta_{\mathcal{L}}^{(2)}=a_1X_4^3+a_2X_6^2+a_3X_{12}+a_4Y_{12},
\]
with $a_i\in\mathbb{Z}\,(1\leq i\leq 4)$. By comparing the Fourier coefficients of both
sides (cf. Example \ref{Igusagenerator} and $\S$\ref{numericalexamples}), we obtain
\[
a_1=1,\quad a_2=0,\quad a_3=43200,\quad a_4=-720.
\]
\end{proof}
\subsection{Another congruence}
\label{anothercong}
In this subsection, we introduce another congruence satisfied by the theta series
$\vartheta_{\mathcal{L}}^{(2)}$.
\begin{Thm}
\label{anothercongruence}
The following congruence relation holds.
\[
\vartheta_{\mathcal{L}}^{(2)} \equiv 1 \pmod{13}.
\]
\end{Thm}
\begin{proof}
It is known that the weight 12 Siegel Eisenstein series $E_{12}^{(2)}$ has the property
\[
E_{12}^{(2)} \equiv 1 \pmod{13}.
\]
On the other hand, we can confirm that
\[
a(\vartheta_{\mathcal{L}}^{(2)};o_2)=1\quad \text{and}\quad a(\vartheta_{\mathcal{L}}^{(2)};T)
\equiv 0 \pmod{13}
\]
for $0\leq T\in\Lambda_2$ with $\text{tr}(T)\leq 6$. This shows that
\[
\vartheta_{\mathcal{L}}^{(2)} \equiv E_{12}^{(2)} \equiv 1 \pmod{13}.
\]
\end{proof}
\begin{Rem}
Of course, the congruence in Theorem \ref{anothercongruence}, means that
\[
a(\vartheta_{\mathcal{L}}^{(2)};T) \equiv 0 \pmod{13}
\]
for any $0_2\ne T\in\Lambda_2$.
\end{Rem}

\section{Numerical Examples}
\label{numericalexamples}
In this section, we present numerical examples of the Fourier coefficients
of $\vartheta_{\mathcal{L}}^{(2)}$ and $\vartheta_{[2,1,3]}^{(2)}$, which
is used in our proof of the main results.
\begin{Ex} Fourier expansion of $\vartheta_{\mathcal{L}}^{(2)}$\\
\label{FourierLeech}
\vspace{2mm}
\\
$\vartheta_{\mathcal{L}}^{(2)}=1+196560(q_{11}^2+q_{22}^2)$
\vspace{2mm}
\\
$+16773120(q_{11}^3+q_{22}^3)$
\vspace{2mm}
\\
$+398034000(q_{11}^4+q_{22}^4)$
\vspace{2mm}
\\
$+(18309564000+9258762240\cdot c_1+904176000\cdot c_2+196560\cdot c_4)q_{11}^2q_{22}^2$
\vspace{2mm}
\\
$+4629381120(q_{11}^5+q_{22}^5)$
\vspace{2mm}
\\
$+(1273079808000+815173632000\cdot c_1+187489935360\cdot c_2+9258762240\cdot c_3)$
\vspace{2mm}
\\
$\cdot(q_{11}^3q_{22}^2+q_{11}^2q_{22}^3)$
\vspace{2mm}
\\
$+34417656000(q_{11}^6+q_{22}^6)$
\vspace{2mm}
\\
$+(26182676520000+18748993536000\cdot c_1+6444966528000\cdot c_2+815173632000\cdot c_3+18309564000\cdot c_4)
(q_{11}^4q_{22}^2+q_{11}^2q_{22}^4)$
\vspace{2mm}
\\
$+(88768382976000+65996457246720\cdot c_1+25779866112000\cdot c_2
 +4320755712000\cdot c_3+187489935360\cdot c_4+16773120\cdot c_6)q_{11}^3q_{22}^3+\cdots$,\\
\\
where $c_i:=q_{12}^i+q_{12}^{-i}$.
\end{Ex}
\begin{Tab} Fourier coefficients of $\vartheta_{\mathcal{L}}^{(2)}$ and $\vartheta_{[2,1.3]}^{(2)}$
\label{coefficients}
\begin{table}[hbtp]
\begin{tabular}{cccc}
$\boldsymbol{T=[m,r,n]}$  &  $\textbf{tr}\boldsymbol{(T)}$  &  $\boldsymbol{ a(\vartheta_{[2,1,3]}^{(2)};T)}$
                                                        &  $\boldsymbol{ a(\vartheta_{\mathcal{L}}^{(2)};T)}$\\ [2mm]\hline\hline
[0,0,0]     &        0      &         1       &   1 \\ \hline
[1,0,0]     &        1      &         0        &   0 \\ \hline
[2,0,0]     &        2      &         2        &   $196560=2^4\cdot 3^3\cdot 5\cdot 7\cdot 13$ \\ \hline
[1,1,1]     &        2      &         0        &   0 \\ \hline
[1,0,1]     &        2      &         0        &   0 \\ \hline
[3,0,0]     &        3      &         2        &   $16773120=2^{12}\cdot 3^2\cdot 5\cdot 7\cdot 13$ \\ \hline
[2,2,1]     &        3      &         0        &   0 \\ \hline
[2,1,1]     &        3      &         0        &   0 \\ \hline
[2,0,1]     &        3      &         0        &   0 \\ \hline
[4,0,0]     &        4      &         2        &   $398034000=2^4\cdot 3^7\cdot 5^3\cdot 7\cdot 13$ \\ \hline
[3,3,1]     &        4      &         0        &   0 \\ \hline
[3,2,1]     &        4      &         0        &   0 \\ \hline
[3,1,1]     &        4      &         0        &   0 \\ \hline
[3,0,1]     &        4      &         0        &   0 \\ \hline
[2,4,2]     &        4      &         2        &   $196560=2^4\cdot 3^3\cdot 5\cdot 7\cdot 13$ \\ \hline
[2,3,2]     &        4      &         0        &   0 \\ \hline
[2,2,2]     &        4      &         0        &   $904176000=2^7\cdot 3^3\cdot 5^3\cdot 7\cdot 13\cdot\underline{23}$ \\ \hline
[2,1,2]     &        4      &         0        &   $9258762240=2^{15}\cdot 3^3\cdot 5\cdot 7\cdot 13\cdot\underline{23}$ \\ \hline
[2,0,2]     &        4      &         0        &   $18309564000=2^5\cdot 3^7\cdot 5^3\cdot 7\cdot 13\cdot\underline{23}$ 
                                                   \\ \hline
[5,0,0]     &        5      &         0        &   $4629381120=2^{14}\cdot 3^3\cdot 5\cdot 7\cdot 13\cdot \underline{23}$
                                                   \\ \hline
[4,4,1]     &        5      &         0        &   0 \\ \hline
[4,3,1]     &        5      &         0        &   0 \\ \hline
[4,2,1]     &        5      &         0        &   0 \\ \hline
[4,1,1]     &        5      &         0        &   0 \\ \hline
[4,0,1]     &        5      &         0        &   0 \\ \hline
[3,4,2]     &        5      &         0        &   0 \\ \hline
[3,3,2]     &        5      &         0        &   $9258762240=2^{15}\cdot 3^3\cdot 5\cdot 7\cdot 13\cdot \underline{23}$
                                                   \\ \hline
[3,2,2]     &        5      &         0        &   $187489935360=2^{13}\cdot 3^7\cdot 5\cdot 7\cdot 13\cdot \underline{23}$
                                                   \\ \hline
[3,1,2]     &        5      &         2        &   $815173632000=2^{15}\cdot 3^7\cdot 5^3\cdot 7\cdot 13$
                                                   \\ \hline
[3,0,2]     &        5      &         0        &   $1273079808000=2^{14}\cdot 3^3\cdot 5^3\cdot 7\cdot 11\cdot 13\cdot 
                                                                \underline{23}$ \\ \hline
[6,0,0]     &        6      &         2        &   $34417656000=2^6\cdot 3^3\cdot 5^3\cdot 7\cdot 13\cdot 17\cdot 103$
                                                   \\ \hline
[5,4,1]     &        6      &         0        &   0 \\ \hline
[5,3,1]     &        6      &         0        &   0 \\ \hline
[5,2,1]     &        6      &         0        &   0 \\ \hline
[5,1,1]     &        6      &         0        &   0 \\ \hline
[5,0,1]     &        6      &         0        &   0 \\ \hline
[4,5,2]     &        6      &         0        &   0 \\ \hline
[4,4,2]     &        6      &         0        &   $18309564000=2^5\cdot 3^7\cdot 5^3\cdot 7\cdot 13\cdot\underline{23}$ 
                                                   \\ \hline
[4,3,2]     &        6      &         2        &   $815173632000=2^{15}\cdot 3^7\cdot 5^3\cdot 7\cdot 13$
                                                   \\ \hline       
[4,2,2]     &        6      &         0        &   $6444966528000=2^{10}\cdot 3^7\cdot 5^3\cdot 7\cdot 11\cdot 13\cdot
                                                   \underline{23}$ \\ \hline
[4,1,2]     &        6      &         0        &   $18748993536000=2^{15}\cdot 3^7\cdot 5^3\cdot 7\cdot 13\cdot
                                                   \underline{23}$ \\ \hline
[4,0,2]     &        6      &         0        &   $26182676520000=2^6\cdot 3^7\cdot 5^4\cdot 7\cdot 11\cdot 13^2\cdot
                                                   \underline{23}$ \\ \hline
[3,6,3]     &        6      &         2        &   $16773120=2^{12}\cdot 3^2\cdot 5\cdot 7\cdot 13$ \\ \hline
[3,5,3]     &        6      &         0        &   0 \\ \hline
[3,4,3]     &        6      &         0        &   $187489935360=2^{13}\cdot 3^7\cdot 5\cdot 7\cdot 13\cdot \underline{23}$
                                                   \\ \hline
[3,3,3]     &        6      &         0        &   $4320755712000=2^{18}\cdot 3^2\cdot 5^3\cdot 7^2\cdot 13\cdot
                                                   \underline{23}$ \\ \hline
[3,2,3]     &        6      &         0        &   $25779866112000=2^{12}\cdot 3^7\cdot 5^3\cdot 7\cdot 11\cdot 13\cdot
                                                   \underline{23}$ \\ \hline
[3,1,3]     &        6      &         0        &   $65996457246720=2^{18}\cdot 3^7\cdot 5\cdot 7\cdot 11\cdot 13\cdot
                                                   \underline{23}$ \\ \hline
[3,0,3]     &        6      &         0        &   $88768382976000=2^{14}\cdot 3^3\cdot 5^3\cdot 7\cdot 13^2\cdot
                                                   \underline{23}\cdot 59$ \\ \hline
\end{tabular}
\end{table}
\end{Tab}


Shoyu Nagaoka\\
Department of Mathematics\\
Kinki University\\
Higashi-Osaka, Osaka 577-8502, Japan\\
Email:nagaoka@math.kindai.ac.jp\\
and\\
Sho Takemori\\
Department of Mathematics\\
Hokkaido University\\
Kita 10, Nishi 8, Kita-Ku\\
Sapporo, Hokkaido, 060-0810, Japan\\
E-mail: stakemorii@gmail.com

\textbf{Acknowledgment}.  The second author is partially supported by the
Grants-in-aid (S) (No. 23224001).

\end{document}